\documentclass[a4,11pt]{article}

\usepackage{amssymb,amsfonts,amsmath,latexsym,color, amscd} 
\usepackage{url}

\usepackage[dvips]{psfrag,graphicx} 

\usepackage{cases}

\newtheorem{theorem}{Theorem}[section]
\newtheorem{lemma}[theorem]{Lemma}
\newtheorem{corollary}[theorem]{Corollary}
\newtheorem{proposition}[theorem]{Proposition}

\newtheorem{definition}[theorem]{Definition}

\newenvironment{proof}{{\par\addvspace{0.1cm}\noindent \bf Proof. }}{\hfill$\Box$\par\medskip}

\newtheorem{remark}[theorem]{Remark}

\numberwithin{equation}{section}

\def\e{\varepsilon}
\def\RR{\mathbb{R}}
\def\CC{\mathbb{C}}

\def\ZZ{\mathbb Z}

\def\R{\Re\mathfrak{e} \,}

\def\ks{\mathcal{K}}
\def\Pf{\mbox{\rm Pf.}}
\def\M{M\"obius }
\def\Mi{M\"obius invariant }
\def\Mt{M\"obius transformation }

\begin{document}

\title{M\"obius invariant metrics on the space of knots}

\author{Jun O'Hara\footnote{Supported by JSPS KAKENHI Grant Number 16K05136.}}
%
%
\maketitle

\begin{abstract}
We give a condition for a function to produce a M\"obius invariant weighted inner product on the tangent space of the space of knots, and show that some kind of M\"obius invariant knot energies can produce M\"obius invariant and parametrization invariant weighted inner products. 
They would give a natural way to study the evolution of knots in the framework of M\"obius geometry. 
\end{abstract}



\section{Introduction}
The motivation of endowing a M\"obius invariant metric to the space of knots originated from the study of the energy $E$ of knots (\ref{defMobiusenergy}) which was given in \cite{O1}. The purpose of introducing the energy is to produce a representative configuration for each knot type as an energy minimizer in the knot type. It turns out to be the beginning of so-called geometric knot theory. 
Freedman, He and Wang gave the gradient of $E$ (\cite{FHW}). 
The regularity of $E$-critical knots has been studied by Zheng-Xu He \cite{He} and later by researchers of analysis including Simon Blatt, Philipp Reiter, Armin Schikorra, Aya Ishizeki, Takeyuki Nagasawa, Alexandra Gilsbach, Heiko von der Mosel, and Nicole Vorderobermeier \cite{B11}, \cite{R}, \cite{IN14}, \cite{IN15}, \cite{BRS}, \cite{Gv}, \cite{BV}, where the gradient is the first step of the study. 

One of the important properties of the energy $E$ is the \M invariance, which is the reason why $E+4$ (\ref{def_Mobius_energy_FHW}) is sometimes called ``{\em M\"obius energy}\,''. It was proved by Freedman, He and Wang \cite{FHW}, and using it they showed that there is an energy minimizer in each prime knot type. 

Although the energy $E$ is M\"obius invariant, the gradient is not. 
Freedman, He and Wang made a comment ``\,Oded Schramm observed that if $\gamma$ is not already a round circle, there exists a M\"obius-invariant inner product $\langle \,,\,\rangle$ on the tangent bundle of $\RR^3$ restricted to $\gamma\,$'' (\cite{FHW} page 41), although the description that follows (in particular, $\langle {\rm d} t, {\rm d} t\rangle_{\textrm{dual}}^{-1/2}$ in the formula (6.16)) is not clear. 

In this paper we give a necessary and sufficient condition for a weight function to produce a M\"obius invariant inner product on the tangent spaces of the space of knots which are not round circles. 
In particular, we will be concerned with a weight function that is independent of the parametrization of knots since we are interested in the shapes of knots i.e. the images of the embeddings. 
Our first example of a weight function is the cube of the {\em regularization of $r^{-2}$-potential}, which is the integrand of the energy $E$. 
The construction can be generalized to obtain weight functions from \Mi $2$-forms that are integrands of \Mi energies of knots. 

Our method would provide a natural way to study the evolution of knots with respect to a \M invariant functional in the quotient space of the shapes of knots modulo the \M group action.

\section{\M invariance condition for weight functions}
We use the following notation in this paper. 
Let $S^1$ be $\RR/\ZZ$, 
$f\colon S^1\hookrightarrow \RR^3$ a knot, $(\cdot \, , \cdot )$ the standard inner product, and 
$T$ a M\"obius transformation of $\RR^3\cup\{\infty\}$. 
Let $\ks$ be the space of knots and $T_f\ks$ the tangent space, 
\[
\begin{array}{rcl}
\ks&=&\displaystyle  \left\{f\colon S^1\hookrightarrow \RR^3\,:\,\mbox{$f$ is injective and }|f'(t)|\ne0 \>\>(\forall t)\right\},
\\[2mm]
T_f\ks&=&\displaystyle \left\{u\colon S^1\hookrightarrow \RR^3\,:\, 
u(t)\perp f'(t) \>\> (\forall t\in S^1)\right\}.
\end{array}
\]
We will choose a suitable differentiability for $f$ in $\ks$ and $u$ in $T_f\ks$ according to the functional we study. For example, if we study the energy $E$ then we can use a Sobolev space $H^3(\RR/\ZZ, \RR^3)$ after \cite{B11} so that both $E$ and the gradient of $E$ are well-defined and a solution of the Euler-Lagrange equation exists. 

The $L^2$ metric on $T_f\ks$ is given by 
\begin{equation}\label{L2-inner_product}
(u,v)_f=\int_{S^1}(u(t),v(t))\,|f'(t)|\, dt \hspace{0.6cm}(u,v \in T_f\ks). 
\end{equation}

We say that a functional $e$ on $\ks$ has a gradient $G_e$ if for any $f\in\ks$ $G_e(f)\in T_f\ks$ satisfies 
\begin{equation}\label{f_L^2-gradient}
\left.\frac{d}{d\varepsilon}e(f+\varepsilon u)\right|_{\varepsilon=0}=(u, G_e(f))_f 
\end{equation}
for any $u\in T_f\ks$. 

We say that a functional $e$ on $\ks$ is \Mi if for any $f\in\ks$ and for any \Mt $T$ such that $T(f(S^1))$ is compact $e(T\circ f)=e(f)$ holds. 

\begin{proposition}\label{prop_round_circles_critical}
Suppose a functional $e$ on $\ks$ is M\"obius invariant and has a gradient $G_e$ in $T_f\ks$. Then round circles are critical with respect to $e$, namely $G_e$ vanishes at round circles.
\end{proposition}

\begin{proof}
Let $f_0\colon S^1\to\RR^2\subset\RR^3$ be a unit circle with center the origin and let $u\in T_{f_0}\ks$. Put $f_\varepsilon=f_0+\varepsilon u$ for small $|\varepsilon|$. 
Let $I$ be an inversion in a unit sphere with center the origin and let $R$ be a reflection in $\RR^2\subset\RR^3$. Put $\widetilde f_\varepsilon=R\circ I\circ f_\varepsilon$. 
We first show 
\begin{equation}\label{tangent_inversion}
\left.\frac{\partial}{\partial\varepsilon}\,\widetilde f_\varepsilon(t)\right|_{\varepsilon=0}
=-\left.\frac{\partial}{\partial\varepsilon}\, f_\varepsilon(t)\right|_{\varepsilon=0}
\hspace{0.8cm}(\forall t\in S^1).
\end{equation} 
Let $\Pi_t$ be a plane through the origin and $f_0(t)$ which is perpendicular to $f'_0(t)$. 
It contains both $f_\varepsilon(t)$ and $\widetilde f_\varepsilon(t)$. 
Let $(x(\varepsilon), z(\varepsilon))$ and $(\widetilde x(\varepsilon), \widetilde z(\varepsilon))$ be the coordinates of $f_\varepsilon(t)$ and $\widetilde f_\varepsilon(t)$ in $\Pi_t$ with respect to the axes given by $\overrightarrow{0f_0(t)}$ and the $z$-axis. Then 
\[
\left(\widetilde x(\varepsilon), \widetilde z(\varepsilon)\right)=
\left(\frac{x(\varepsilon)}{x(\varepsilon)^2+z(\varepsilon)^2},\, -\frac{z(\varepsilon)}{x(\varepsilon)^2+z(\varepsilon)^2}
\right),
\]
and therefore, substituting $x(0)=1$ and $z(0)=0$ we obtain 
$\left({\widetilde x\,}'(0), {\widetilde z\,}'(0)\right)=-(x'(0), z'(0))$, which means \eqref{tangent_inversion}. 

As a result, using $\widetilde f_0=f_0$ we have 
\[
\begin{array}{rcl}
\displaystyle \left.\frac{d}{d\varepsilon}e\left(\widetilde f_\varepsilon\right)\right|_{\varepsilon=0} &=&
\displaystyle \int_{S^1}\left(\left.\frac{\partial}{\partial\varepsilon}\,\widetilde f_\varepsilon(t)\right|_{\varepsilon=0},\,G_e\big(\widetilde f_0\big)(t)\right)\,\left|{\widetilde f_0{}}'(t)\right|\,dt \\[5mm]
&=& \displaystyle \int_{S^1}\left(-\left.\frac{\partial}{\partial\varepsilon}\, f_\varepsilon(t)\right|_{\varepsilon=0},\,G_e(f_0)(t)\right)\,\left|{f_0{}}'(t)\right|\,dt \\[5mm]
&=& \displaystyle -\left.\frac{d}{d\varepsilon}e\left(f_\varepsilon\right)\right|_{\varepsilon=0}\,.
\end{array}
\]
Since $e\big(\widetilde f_\varepsilon\big)=e\left(f_\varepsilon\right)$ by the \M invariance of $e$, the above implies 
\[
\left.\frac{d}{d\varepsilon}e\left(f_\varepsilon\right)\right|_{\varepsilon=0}=0\,, 
\]
hence 
\[
0=\left.\frac{d}{d\varepsilon}e(f_0+\varepsilon u)\right|_{\varepsilon=0}
=\int_{S^1}\left(u(t),\,G_e(f_0)(t)\right)\,\left|{f_0{}}'(t)\right|\,dt 
\hspace{0.6cm}(\forall u\in T_{f_0}\ks), 
\]
which implies $G_e(f_0)\equiv0$. 
\end{proof}

Even if a functional is \M invariant, the gradient is not, namely, $T_\ast G_e(f)\ne G_e(T\circ f)$. This is because 
\[
(T_\ast u, T_\ast v)_{T\circ f}\ne (u,v)_f 
\]
in general. For example, when $T$ is a homothety by a factor $k>0$ $(T_\ast u, T_\ast v)_{T\circ f}=k^3 (u,v)_f$. 

\smallskip
Let us look for a weighted inner product that is compatible with \M transformations and that is independent of the parametrization of knots. Proposition \ref{prop_round_circles_critical} implies that we can restrict ourselves to the weight functions defined on $\ks^{\,\circ}\times S^1$, where $\ks^{\,\circ}$ is the set of non-circular knots. Our weight function $\Phi$ should satisfy 
the following conditions: 
\begin{enumerate}
\item[(i)] $\Phi$ is positive on $\ks^{\,\circ}\times S^1$, and $\Phi(f, t)$ is continuous in $t\in S^1$ for each $f\in\ks^{\,\circ}$.
\item[(ii)] $\Phi$ is parametrization independent, i.e., for any diffeomorphism $\rho$ of $S^1$, $\Phi(f\circ\rho, t)=\Phi(f,\rho(t))$ for any $t\in S^1$. 
\item [(iii)] The $\Phi$-weighted inner product 
\begin{equation}\label{f_inner_product_varphi}
\langle u,v\rangle_\Phi=\int_{S^1}(u(t),v(t))\,\Phi(f,t)\,|f'(t)|\, dt \hspace{0.6cm}(u,v \in T_f\ks)
\end{equation}
is M\"obius invariant, i.e., 
\begin{equation}\label{Mobius_equivariance_inner_product}
\langle T_\ast u, T_\ast v\rangle_{\Phi}=\langle u,v\rangle_{\Phi} \hspace{0.6cm}(\forall u,v \in T_f\ks).
\end{equation}
\end{enumerate}

\smallskip
We recall some basics of M\"obius geometry. 
Let $T$ be a M\"obius transformation of $\RR^3\cup\{\infty\}$. 
For a point $p$ in $\RR^3$ put 
\[\left|T'(p)\right|={|\det DT(p)|}^{1/6},\] 
where $DT$ is the Jacobian matrix of $T$. 
In particular, if $T$ is a homothety by $k>0$ then $\left|T'(p)\right|=\sqrt k$, and 
if $T$ is an inversion in a sphere of radius $r$ with center $C$ then 
\begin{equation}\label{T'_inversion}
\left|T'(p)\right| =\frac r{\,|C-p|\,}\,. 
\end{equation}
Note that we have 
\begin{eqnarray}
|T(p)-T(q)|&=&\displaystyle \left|T'(p)\right|\,\left|T'(q)\right|\,|p-q| \hspace{0.4cm} (p,q\in\RR^3), \label{formula_T_distance} \\
\displaystyle \left|(T\circ \gamma)'(t)\right|&=&\displaystyle {\left|T'(\gamma(t))\right|}^2\,|\gamma'(t)| \hspace{1.0cm} (\gamma\colon J\to\RR^3, \mbox{$J$ is an interval}). \label{formula_T_tangent_vector}
\end{eqnarray}

We first investigate without assuming the parametrization independence of $\Phi$. 

\begin{theorem}\label{thm1} 
{\rm (1)} 
Suppose $\Phi\colon \ks^{\,\circ}\times S^1\to\RR$ satisfies the condition {\rm (i)} above. 
Then the $\Phi$-weighted inner product \eqref{f_inner_product_varphi} is M\"obius invariant i.e. \eqref{Mobius_equivariance_inner_product} holds if and only if $\Phi$ satisfies 
\begin{equation}\label{condition_varphi}
\Phi(T\circ f,t)={\left|T'(f(t))\right|}^{-6}\Phi(f,t)
\end{equation}
for any $f$ in $\ks$, $t$ in $S^1$, and for any M\"obius transformation $T$ such that $T(f(S^1))$ is compact. 

\medskip
{\rm (2)} Suppose $\Phi$ satisfies the condition  {\rm (i)} and \eqref{condition_varphi}. 
Suppose a functional $e$ on $\ks$ is M\"obius invariant and has a gradient $G_e$ in $T_f\ks$. 
Define the $\Phi^{-1}$-weighted gradient $\mathcal{G}^\Phi_e(f)$ in $T_f\ks$ by 
$\mathcal{G}^\Phi_e(f)(t)=(\Phi(f,t))^{-1}G_e(f)(t)$ for a non-circular knot $f$ in $\ks^{\,\circ}$ and $\mathcal{G}^\Phi_e(f_0)=0$ for a round circle $f_0$. Then for any $f$ in $\ks$ 
\begin{eqnarray}
\displaystyle \left.\frac{d}{d\varepsilon}e(f+\varepsilon u)\right|_{\varepsilon=0}&=&\displaystyle \left\langle u, \mathcal{G}^\Phi_e(f) \right\rangle_\Phi\, \hspace{0.6cm}(u\in T_f\ks).  \label{f_Mobius_equiv_Phi_gradient1} \\ [2mm]
\mathcal{G}^\Phi_e(T\circ f)&=&T_\ast\left(\mathcal{G}^\Phi_e(f)\right).  \label{f_Mobius_equiv_Phi_gradient2}
\end{eqnarray}
Namely, $\mathcal{G}^\Phi_e$ is a \Mi gradient of $e$ with respect to $\Phi$-weighted inner product. 
\end{theorem}

\begin{proof}
(1) The formula \eqref{formula_T_tangent_vector} implies 
\[
\begin{array}{rcl}
\left(T_\ast u(t), T_\ast v(t)\right)&=&
\displaystyle \frac12\left[{\left|T_\ast (u+v)(t)\right|}^2-{\left|T_\ast u(t)\right|}^2-{\left|T_\ast v(t)\right|}^2\right] \\[4mm]
&=&\displaystyle \frac12 \, {\left|T'(f(t))\right|}^4 \left[{\left|(u+v)(t)\right|}^2-{\left|u(t)\right|}^2-{\left|v(t)\right|}^2\right] \\[4mm]
&=&\displaystyle {\left|T'(f(t))\right|}^4 (u(t),v(t)), \\[4mm]
\displaystyle \left|(T\circ f)'(t)\right|
&=&{\left|T'(f(t))\right|}^2\, |f'(t)|.
\end{array}
\]
It follows that 
\[
\begin{array}{rcl}
\displaystyle \langle T_\ast u, T_\ast v\rangle_\Phi 
&=&
\displaystyle \int_{S^1} \left(T_\ast u(t), T_\ast v(t)\right)\Phi(T\circ f,t)\left|(T\circ f)'(t)\right|\,dt \\[4mm]
&=&\displaystyle \int_{S^1} \left(u(t), v(t)\right){\left|T'(f(t))\right|}^6\Phi(T\circ f,t)\left|f'(t)\right|\,dt. 
\end{array}
\]
Therefore, $\langle T_\ast u, T_\ast v\rangle_\Phi=\langle u,v\rangle_\Phi$ for any $u$ and $v$ in $T_f\ks^{\,\circ}$ if and only if ${\left|T'(f(t))\right|}^6\Phi(T\circ f,t)=\Phi(f,t)$ for all $t$.

\medskip
(2) The equality \eqref{f_Mobius_equiv_Phi_gradient1} follows from \eqref{f_L^2-gradient}, \eqref{L2-inner_product}, \eqref{f_inner_product_varphi}, and the definition of $\mathcal{G}^\Phi_e$. 

The equality \eqref{f_Mobius_equiv_Phi_gradient2} follows from the uniqueness of the gradient as follows. 
First note that since a M\"obius transformation $T$ is conformal, $T_\ast\left(\mathcal{G}^\Phi_e(f)(t)\right)$ is perpendicular to $T_\ast f'(t)$, and hence $T_\ast\left(\mathcal{G}^\Phi_e(f)\right)\in T_{T\circ f}\ks$. 
By Proposition \ref{prop_round_circles_critical} we may assume that $f$ and $T\circ f$ are not round circles. 

Let $u\in T_f\ks$. The equality \eqref{f_Mobius_equiv_Phi_gradient1} and the \M invariance of $\Phi$-weighted inner product imply 
\begin{equation}\label{f1}
\left.\frac{d}{d\varepsilon}e(f+\varepsilon u)\right|_{\varepsilon=0}=\left\langle u, \mathcal{G}^\Phi_e(f) \right\rangle_\Phi
=\left\langle T_\ast u, T_\ast(\mathcal{G}^\Phi_e(f))\right\rangle_{\Phi}\,.
\end{equation}
On the other hand, since $e(f+\varepsilon u)=e(T\circ (f+\varepsilon u))$ by the M\"obius invariance of $e$ and since 
\[
\left.\frac{\partial}{\partial\varepsilon}\left(T\circ(f+\varepsilon u)\right)(t)\right|_{\varepsilon=0}=T_\ast u(t),
\]
\begin{eqnarray}\label{}
\displaystyle \left.\frac{d}{d\varepsilon}e(f+\varepsilon u)\right|_{\varepsilon=0}
&=& \displaystyle \left.\frac{d}{d\varepsilon}e(T\circ(f+\varepsilon u))\right|_{\varepsilon=0} \nonumber \\ [2mm]
&=& \displaystyle  \displaystyle \left.\frac{d}{d\varepsilon}e(T\circ f+\varepsilon\, T_\ast u)\right|_{\varepsilon=0} \nonumber \\[1mm]
&=& \displaystyle \left\langle T_\ast u, \mathcal{G}^\Phi_e(T\circ f)\right\rangle_{\Phi}\,, \label{f2}
\end{eqnarray}
where the last equality follows from \eqref{f_Mobius_equiv_Phi_gradient1}. 
As $u$ is arbitrary, from \eqref{f1} and \eqref{f2} we have $\mathcal{G}^\Phi_e(T\circ f)=T_\ast(\mathcal{G}^\Phi_e(f))$. 
\end{proof}

\begin{remark} The equation \eqref{formula_T_tangent_vector} implies that if we put $\Phi_0(f,t)={|f'(t)|}^{-3}$ then it satisfies the condition \eqref{condition_varphi}, hence 
\[
\langle u,v\rangle_{\Phi_0}=\int_{S^1}\frac{(u(t),v(t))}{{|f'(t)|}^2}\,dt
\]
is \M invariant, although $\Phi_0$ is not parametrization independent. 
\end{remark}

Suppose $\Phi\colon \ks^{\,\circ}\times S^1\to\RR$ satisfies the condition  {\rm (i)} and \eqref{condition_varphi}. Then  
a functional $E_{\Phi^{1/3}}$ on $\ks^{\,\circ}$ given by 
\[
E_{\Phi^{1/3}}(f)=\int_{S^1}\sqrt[3]{\Phi(f,t)}\,|f'(t)|\,dt
\]
is \M invariant.  
From a viewpoint of geometric knot theory, it would be desirable if $\Phi$ is parametrization independent, $\Phi$ can be continuously extended to a non-negative function on $\ks\times S^1$, and the global minimum of $E_{\Phi^{1/3}}$ is given by round circles. 
The examples we introduce in the following two sections enjoy the above properties.

\section{Weight function given by a potential of the \M energy} 
We give examples of $\Phi$ that satisfy the conditions (i), (ii) and (iii) in what follows. 
First example is the simplest one, and it seems the most appropriate to be applied to the study the energy $E$. 

\begin{definition}\label{def_V} \rm (\cite{O1}) 
Define the {\em regularized $r^{-2}$-potential} 
of a knot $f\in\mathcal{K}$ at a point $f(s)$ by 
\begin{equation}\label{f_def_V}
V(f,s)=\lim_{\varepsilon\downarrow0}\left(\int_{d_f(s,t)\ge\varepsilon}\,\frac{|f'(t)|\,dt}{{|f(t)-f(s)|}^2}-\frac{\,2\,}\varepsilon\right),
\end{equation}
where $d_f(s,t)$ is the arc-length between $f(s)$ and $f(t)$ along the knot $f(S^1)$. 
\end{definition}

Obviously it is parametrization independent, i.e., $V(f\circ\rho,s)=V(f,\rho(s))$ for any diffeomorphism $\rho$ of $S^1$. 

We introduce so-called ``wasted length argument'' by Doyle and Schramm reported in \cite{KS} (with a tiny modification), which is an epochal observation in the \M geometric study of knots. 
Let $I$ be an inversion in a unit sphere with center $f(s)$. Let $\widetilde f_s=I\circ f$ be an inverted open knot. Then \eqref{formula_T_tangent_vector} implies that the integrand of $V(f,s)$, 
\[
\frac{|f'(t)|\,dt}{{|f(t)-f(s)|}^2}
\]
is a length element of the inverted open knot. 
This observation implies 
\[
V(f,s)=\lim_{s_1\to s^-, \,s_2\to s^+}\left(d_{\widetilde f_s}(s_1, s_2)-\left|\widetilde f_s(s_1)-\widetilde f_s(s_2)\right|\right),
\]
where $f(s_1)$ and $f(s_2)$ approach $f(s)$ from opposite sides. 
In fact, when $d_f(s,s_1)=d_f(s,s_2)=\varepsilon$ the above equality follows from the fact that $\big|\widetilde f_s(s_1)-\widetilde f_s(s_2)\big|=2/\varepsilon+O(\varepsilon)$. 
Remark $d_{\widetilde f_s}(s_1, s_2)$ is the arc-length of $\widetilde f_s(S^1\setminus (s_1, s_2))$. 
Thus $V(f,s)$ can be interpreted as the difference of the lengths of the open knot inverted at $f(s)$ and the asymptotic line. 
Therefore, $V(f,s)$ is non-negative and is equal to $0$ if and only if the image of $\widetilde f_s$ is a line, which occurs if and only if $f(S^1)$ is a round circle. 

This argument yields the cosine formula of $V$; 
\begin{equation}\label{cosine_formula_V}
V(f,s)=\int_{S^1}\frac{1-\cos\theta_f(s,t)}{{|f(s)-f(t)|}^2}\,|f'(t)|\,dt,
\end{equation}
where $\theta_f(s,t)$ $(0\le\theta_f\le\pi)$ is the angle between two circles, both through $f(s)$ and $f(t)$, one is tangent to the knot at $f(s)$ ad the other at $f(t)$. We call $\theta_f$ the {\em conformal angle}. 
Since $\theta_f$ is defined by using only circles, tangency and angle, it is \M invariant.  

The conformal angle $\theta_f$ is of the order of $|s-t|^2$ near the diagonal set $\Delta=\{(t,t)\,:\,t\in S^1\}$ (\cite{LO}). Therefore, the integrand of \eqref{cosine_formula_V} can be extended to a continuous function on $S^1\times S^1$, which implies that $V(f,s)$ is continuous in $s$. 

Summarizing it up; 

\begin{proposition}\label{Doyle-Schramm} The potential $V(f,s)$ is continuous in $s$, and $V(f,s)\ge0$ for any $f$ in $\mathcal{K}$ and for any $s$ in $S^1$, where the equality holds if and only if $f(S^1)$ is a round circle. 
\end{proposition}

The {\em energy} $E$ (\cite{O1}) is given by 
\begin{equation}\label{defMobiusenergy}
\begin{array}{rcl}
E(f)&=&\displaystyle \int_{S^1}V(f,s)\,|f'(s)|\,ds \\[4mm]
&=&\displaystyle \lim_{\varepsilon\downarrow0}\left(\iint_{S^1\times S^1, d_f(s,t)\ge\varepsilon} 
\,\frac{|f'(s)|\,|f'(t)|\,dsdt}{{|f(t)-f(s)|}^2}
- \frac{2L(f(S^1))}\varepsilon\right),
\end{array}
\end{equation}
where $L(f(S^1))$ is the length of the knot. 
Then 
\begin{equation}\label{def_Mobius_energy_FHW}
E(f)+4=\iint_{S^1\times S^1}\left(\frac{1}{{|f(t)-f(s)|}^2}
-\frac1{d_f(s,t)^2} \right) |f'(s)|\,|f'(t)|\,dsdt 
\end{equation}
(\cite{Nak} and \cite{FHW}). The right hand side is called the {\em M\"obius energy} of a knot $f$. 

\begin{lemma} \label{lemma_V}
Let $T$ be a M\"obius transformation such that $T(f(S^1))$ is compact. Then we have 
\begin{equation}\label{V_Mobius}
V({T\circ f},s)={\left|T'(f(s))\right|}^{-2}\,V(f,s).
\end{equation}

\end{lemma}

The lemma follows directly from \eqref{cosine_formula_V} since \eqref{formula_T_distance}, \eqref{formula_T_tangent_vector} and the \M invariance of the conformal angle implies 
\[
\frac{1-\cos\theta_{T\circ f}(s,t)}{{\left|(T\circ f)(s)-(T\circ f)(t)\right|}^2}\,\left|(T\circ f)'(t)\right|
=
\frac{1-\cos\theta_f(s,t)}{{\left|T'(f(s))\right|}^2\,{|f(s)-f(t)|}^2}\,|f'(t)|. 
\]
We introduce an alternative proof using analytic continuation based on the idea of Brylinski \cite{B}. 
Another proof by Hadamard regularization \eqref{f_def_V} will be given in Section \ref{section_Hadamard}. 

\begin{proof}
Define the {\em local Brylinski's beta function} $B_{f,\,s}^{\,\rm loc}(z)$ $(z\in\CC)$ by the meromorphic regularization of
\[
\int_{S^1}{|f(t)-f(s)|}^z\,|f'(t)|\,dt, 
\]
which can be carried out as follows. First note that the integral, considered as a function of $z$, 
is well-defined if $\R z>-1$ and is a holomorphic function of $z$ there. 
We can extend the domain by analytic continuation to the whole complex plane to obtain a meromorphic function only with (possible) simple poles at negative odd integers, $z=-1,-3,\dots$. 
Then the regularized $r^{-2}$-potential satisfies $V(f,s)=B_{f,\,s}^{\,\rm loc}(-2)$ (\cite{B}). 
Then by \eqref{formula_T_distance} and \eqref{formula_T_tangent_vector}, 
\[
\begin{array}{rcl}
B_{T\circ f,\,s}^{\,\rm loc}(z)&=&\displaystyle \int_{S^1}{|T(f(t))-T(f(s))|}^z\,\left|{(T\circ f)}'(t)\right|\,dt \\[4mm]
&=&\displaystyle \int_{S^1} {\left(\,{\left|T'(f(t))\right|}\,{\left|T'(f(s))\right|}\,|f(t)-f(s)|\,\right)}^z\, {\left|T'(f(t))\right|}^2\left|f'(t)\right|\,dt \\[4mm]
&=&\displaystyle {\left|T'(f(s))\right|}^z \int_{S^1} {\left|T'(f(t))\right|}^{z+2}\, {|f(t)-f(s)|}^z\, \left|f'(t)\right|\,dt. 
\end{array}
\]
Substituting $z=-2$ we obtain 
\[
V({T\circ f},s)=B_{T\circ f,\,s}^{\,\rm loc}(-2)={\left|T'(f(s))\right|}^{-2}B_{T\circ f,\,s}^{\,\rm loc}(-2)={\left|T'(f(s))\right|}^{-2}\,V(f,s).
\]
\end{proof}

\begin{corollary}\label{cor}
The cube of the regularized $r^{-2}$-potential, $V^3$, satisfies the conditions {\rm (i), (ii)} and {\rm (iii)}. 
Therefore, 
\[
\langle u,v\rangle_{V^3} =\int_{S^1}(u(t),v(t))\,{V(f,t)}^3\,|f'(t)|\, dt
\]
is a \Mi and parametrization independent inner product, and if $e$ is a \Mi functional on $\ks$ with a gradient $G_e$ then $\mathcal{G}^{V^3}_e(f)=(V(f))^{-3}G_e(f)$ $(f\in\ks^{\,\circ})$ is a \Mi and parametrization independent gradient with respect to $\langle \,\cdot\,,\,\cdot\,\rangle_{V^3}$. 
\end{corollary}

\section{Generalization}
Note that \eqref{cosine_formula_V} induces the cosine formula of the energy $E$ by Doyle and Schramm; 
\begin{equation}\label{cosine_formula}
E(f)=\iint_{S^1\times S^1}\frac{1-\cos\theta_f(s,t)}{{|f(s)-f(t)|}^2}\,|f'(s)|\,|f'(t)|\,dsdt,
\end{equation}
where $\theta_f$ is the conformal angle. Since 
\begin{equation}\label{Mobius_invariant_2-form_coeff}
\frac{|f'(s)|\,|f'(t)|}{{|f(s)-f(t)|}^2}
\end{equation}
is \Mi by \eqref{formula_T_distance} and \eqref{formula_T_tangent_vector}, so is the integrand of \eqref{cosine_formula}. 

\medskip
The construction of $\Phi$ by the potential $V$ can be generalized as follows. 

\begin{theorem}\label{thm_gen}
Let $\psi$ be a function from $\ks\times\left((S^1\times S^1)\setminus\Delta\right)$ to $\RR$, where $\Delta$ is the diagonal set $\Delta=\{(t,t)\,:\,t\in S^1\}$. 
Suppose $\psi(f,s,\,\cdot\,)\,|f'(\,\cdot\,)|$ is integrable on $S^1$ for any $f\in\ks$ and for any $s\in S^1$. Put 
\[
\Psi(f,s)=\int_{S^1}\psi(f,s,t)\,|f'(t)|\,dt.
\]
Suppose $\Psi$ is positive on $\ks^{\,\circ}\times S^1$ and $\Psi(f,s)$ is continuous in $s$ for any $f\in\ks$. 

Then, if the $2$-form $\psi(f,s,t)\,|f'(s)|\,|f'(t)|\,dsdt$ on $S^1\times S^1\setminus\Delta$  is \M invariant, i.e. 
\begin{equation}\label{Mobius_invariance_2-form}
\psi(T\circ f,s,t)\,|(T\circ f)'(s)|\,|(T\circ f)'(t)|=
\psi(f,s,t)\,|f'(s)|\,|f'(t)|
\end{equation}
for any \Mt $T$ such that $T(f(S^1))$ is compact, then $(\Psi)^3$ satisfies the condition \eqref{condition_varphi} for $\Phi$ in Theorem {\rm \ref{thm1}} {\rm (1)}. 
\end{theorem}
\begin{proof} 
By \eqref{formula_T_tangent_vector}, \eqref{Mobius_invariance_2-form} implies 
\[
\Psi(T\circ f,s)\,{\left|T'(f(s))\right|}^2=\Psi(f,s).
\]
\end{proof}

In the above scheme, the weight function studied in the previous section is given by 
\[
\psi(f,s,t)=\frac{1-\cos\theta_f}{{|f(s)-f(t)|}^2}\,. 
\]

Since \eqref{Mobius_invariant_2-form_coeff} is \M invariant, if $\mu(\theta)$ is a continuous function on $[0,\pi]$ that is positive on $(0,\pi)$ and is of the order of $\theta$ (to be precise, $O(\theta^{1/2+\varepsilon})$ for some $\varepsilon>0$) near $\theta=0$, then 
\begin{equation}\label{psi_conf_angle}
\psi(f,s,t)=\frac{\mu(\theta_f(s,t))}{{|f(s)-f(t)|}^2}
\end{equation}
satisfies the condition in Theorem \ref{thm_gen}. 
Furthermore, $\Psi$ is parametrization independent. 

Let us give examples of $\mu(\theta)$ which have been studied in geometric knot theory to produce knot energies by 
\[
E_\mu(f)=\iint_{S^1\times S^1}\mu(\theta_f(s,t))\,\frac{\,|f'(s)|\,|f'(t)|\,}{{|f(s)-f(t)|}^2}\,dsdt
\]
(cf. \cite{KS}, Section 2).
\begin{enumerate}
\item As we have seen, the energy $E$ \eqref{defMobiusenergy} is given by $\mu(\theta)=1-\cos\theta$. 
\item The absolute sine energy (Section 4 of \cite{KS}) or the absolute imaginary cross-ratio energy (Section 5 of \cite{LO}) is given by  $\mu(\theta)=\sin\theta$. Remark our convention $0\le\theta_f\le\pi$ implies $\sin\theta_f\ge0$. 
\item The measure of acyclicity (Section 6 of \cite{LO}) is given by $\mu(\theta)=(\pi/4)(\theta-\theta\sin\theta)$ (Proposition 6.13 of \cite{LO}). 
\end{enumerate}

Note that $\psi$ of the form \eqref{psi_conf_angle} depends only on the first-order information of the knot $f$. 
Conversely, if a \Mi function $\psi(f,s,t)$ depends only on the first-order information of the knot $f$, then $\psi$ should be expressed in terms of \eqref{Mobius_invariant_2-form_coeff} and the conformal angle $\theta_f$ (\cite{LO} page 234). 
In fact, the $2$-form ${|f(s)-f(t)|}^{-2}\,|f'(s)|\,|f'(t)|\,dsdt$ and $\theta_f$ are the absolute value and the argument of the {\em infinitesimal cross ratio} of the knot (\cite{LO}), which is the cross ratio of the four points $f(s), f(s+ds), f(t)$ and $f(t+dt)$, where the four points are considered as complex numbers through a stereographic projection from a sphere through these four points to a plane, which is identified with $\CC$. 
Now the above assertion follows from the fact that the conjugacy class of the cross ratio is essentially the only \M invariant of the set of (ordered) four points in $\RR^3$.

\section{Proof of \M invariance by Hadamard regularization}\label{section_Hadamard} 
The gradient of the energy $E$, $G_E$, was given in \cite{FHW}. 
Corollary \ref{cor} shows that $\mathcal{G}_E^{V^3}(f)=(V(f))^{-3}G_E(f)$ is M\"obius invariant. 
In this section we give an outline of the proof of it by Hadamard regularization. 

Hadamard regularization is a method to obtain a finite value from a divergent integral. Suppose $\int_X\omega$ diverges on $\Delta\subset X$. Integrate $\omega$ on the complement of an $\varepsilon$ neighbourhood of $\Delta$, expand the resulting value in a (Laurent) series of $\varepsilon$, and finally take the constant term, which is called {\em Hadamard's finite part}, denoted by $\Pf\int_X\omega$. 
It is a kind of generalization of Cauchy's principal value. 
\\ \indent
Let us first show Lemma \ref{lemma_V} by Hadamard regularization. 
It is enough to show it when $T$ is an inversion $I$ in a unit sphere with center the origin and $f(S^1)$ does not pass through the origin. 
Since 
\[
V(f,s)=\Pf\int_{S^1}\frac1{|f(t)-f(s)|^2}\,|f'(t)|\,dt,
\]
we have 
\[
V({I\circ f},s)=\Pf\int_{S^1}\frac{|f(s)|^2|f(t)|^2}{|f(t)-f(s)|^2}\,\frac{|f'(t)|}{|f(t)|^2}\,dt=|f(s)|^2\,V(f,s),
\]
which proves \eqref{V_Mobius} since $|I'(f(s))|=|f(s)|^{-1}$ by \eqref{T'_inversion}.  

Let $P_{f'(s)^\perp}:\Bbb R^3\to\Bbb R^3$ be the orthogonal projection to $\left(\mbox{Span}\langle f'(s) \rangle \right)^\perp$. Then $G_E(f)(s)$ is given by
\begin{equation}\label{gradient_FHW}
2\,\mbox{\rm p.v.}\int_{S^1}\left\{2\,\dfrac{P_{f'(s)^\perp}(f(t)-f(s))}
{{|f(t)-f(s)|}^2}
-\dfrac1{|f'(s)|}\,\dfrac d{ds}\left(
\dfrac{f'(s)}{|f'(s)|}\right)\right\} 
\dfrac{|f'(t)|}{{|f(t)-f(s)|}^2}\,dt,
\end{equation}
where {\rm p.v.} means Cauchy's principal value integral (\cite{FHW}). 
Note that 
\[
\dfrac1{|f'(s)|}\,\dfrac d{ds}\left(\dfrac{f'(s)}{|f'(s)|}\right)
=\frac{f''(s)}{{|f'(s)|}^2}-\frac{(f'(s),f''(s))}{{|f'(s)|}^4}f'(s),
\]
and that it is perpendicular to $f'(s)$. 
\\ \indent
In the formulae in what follows we fix $s$ in $S^1$ and may omit $(s)$ from $f(s)$, $f'(s)$, and $f''(s)$ for the sake of visible clarity when the formulae are complicated. 

Using Bouquet's formula that gives the expansion of $f(t)$ with respect to the Frenet frame at $f(s)$ in a series of $t-s$, which can be deduced from Frenet-Serret formulas, we have 
\[
\int_{d_f(t,s)\ge\e}\frac{f(t)-f(s)}{{|f(t)-f(s)|}^4}\,|f'(t)|\,dt
=\frac1\e\left(\frac{f''}{{|f'|}^2}-\frac{(f',f'')}{{|f'|}^4}f'\right)+O(1), 
%
\]
which implies 
\[\begin{array}{l}
\displaystyle \Pf\int_{S^1}\frac{f(t)-f(s)}{{|f(t)-f(s)|}^4}\,|f'(t)|\,dt \\[4mm]
=\displaystyle \lim_{\e\to0^+}\left[
\int_{d_f(t,s)\ge\e}\frac{f(t)-f(s)}{{|f(t)-f(s)|}^4}\,|f'(t)|\,dt
-\frac1\e\left(\frac{f''}{{|f'|}^2}-\frac{(f',f'')}{{|f'|}^4}f'\right)
%
\right].
\end{array}
\]
Together with 
\[
\int_{d_f(t,s)\ge\e}\frac{|f'(t)|}{{|f(t)-f(s)|}^2}\,dt
=V(f,s)+\frac2\e+O(\e), 
\]
we have 
\[\begin{array}{l}
\displaystyle \Pf\int_{S^1}\frac{f(t)-f(s)}{{|f(t)-f(s)|}^4}\,|f'(t)|\,dt \\[4mm]
=\displaystyle \int_{S^1} \left[ \frac{f(t)-f(s)}{{|f(t)-f(s)|}^4}\,|f'(t)|
-\frac{|f'(t)|}{2{|f(t)-f(s)|}^2}\left(\frac{f''}{{|f'|}^2}-\frac{(f',f'')}{{|f'|}^4}f'\right)
\right]dt \\[4mm]
\phantom{=}\displaystyle +\frac12V(f,s)\left(\frac{f''}{{|f'|}^2}-\frac{(f',f'')}{{|f'|}^4}f'\right).
%
%
\end{array}
\]
Taking the image of the above under $P_{f'(s)^\perp}$ and comparing it with \eqref{gradient_FHW}, one obtain
\begin{equation}\label{grad_FHW-Pf}
G_E(f)(s)
=2P_{f'(s)^\perp}\left[
2\,\Pf\int_{S^1}\frac{f(t)-f(s)}{{|f(t)-f(s)|}^4}\,|f'(t)|\,dt
-\frac{V(f,s)}{{|f'(s)|}^2}\,f''(s)
\right]. 
\end{equation}
Put
\begin{eqnarray}\label{}
u_1(f,s)&=&\displaystyle 2\,\Pf\int_{S^1}\frac{f(t)-f(s)}{{|f(t)-f(s)|}^4}\,|f'(t)|\,dt,  \label{u1} \\
u_2(f,s)&=&\displaystyle -\frac{V(f,s)}{{|f'(s)|}^2}\,f''(s), \label{u2} \\
u_3(f,s)&=&\displaystyle -\left(u_1(f,s),\frac{f'(s)}{|f'(s)|}\right)\frac{f'(s)}{|f'(s)|}, \label{u3} \\
u_4(f,s)&=&\displaystyle -\left(u_2(f,s),\frac{f'(s)}{|f'(s)|}\right)\frac{f'(s)}{|f'(s)|}. \label{u4} 
\end{eqnarray}
Then 
\begin{equation}\label{G_E_u1u4}
G_E(f)(s)=2(u_1(f,s)+u_2(f,s)+u_3(f,s)+u_4(f,s)).
\end{equation}
\begin{lemma}\label{lemma_cond_MI}
Let $u$ be a map from $\mathcal{K}\times S^1$ to $\RR^3$ that maps a pair $(f,s)$ to $u(f,s)$ in $T_{f(s)}\RR^3\cong\RR^3$. 
Then $V^{-3}u$ is \M invariant, i.e. 
\[
{V(T\circ f,s)}^{-3}u(T\circ f,s)=T_\ast\left({V(f,s)}^{-3}u(f,s)\right)
\]
for any $f$ in $\mathcal{K}$, $s$ in $S^1$, and a \Mt $T$ such that $T(f(S^1))$ is comapct if and only if the following conditions are satisfied for any $f$ and $s$. 
\begin{enumerate}
\item For any translation $p$ of $\RR^3$, $u(p\circ f,s)=u(f,s)$ holds. 
\item For any homothety by $k>0$, $u(kf,s)=k^{-2}u(f,s)$ holds.  
\item Let $I$ be an inversion in a unit sphere with center the origin. If $f(S^1)$ does not pass through the origin then  
\begin{equation}\label{cond_u_MI}
u(I\circ f,s)-{|f(s)|}^4u(f,s)+2{|f(s)|}^2\left(u(f,s),f(s)\right)f(s)=0
\end{equation}
holds. 
\end{enumerate}
\end{lemma}
Let the left hand side of \eqref{cond_u_MI} be denoted by $J(u,f,s)$. 
\begin{proof}
First note that for $v\in T_{f(s)}\RR^3$ we have 
\[
I_\ast(v)=\frac{v}{{|f|}^2}-2\frac{(v,f)}{{|f|}^4}f. 
%
\]
where we omitted $(s)$ from $f(s)$ as before. 
Put $\tilde f=I\circ f$. 
Then by \eqref{V_Mobius} we have 
\[
\frac{u(\tilde f,s)}{{V(\tilde f,s)}^3}-I_\ast\left(\frac{u(f,s)}{{V(f,s)}^3}\right)
=\frac1{{V(f,s)}^3}\left(
\frac{u(\tilde f,s)}{{|f|}^6}-\frac{u(f,s)}{{|f|}^2}+2\frac{\left(u(f,s),f\right)}{{|f|}^4}f
\right), 
%
\]
which implies \eqref{cond_u_MI}. 
\end{proof}
\begin{lemma}\label{lemma_Ju14}
The following equalities hold, where we omit $(s)$ from $f(s)$ and $f'(s)$. 
\begin{eqnarray}\label{}
J(u_1,f,s)&=&-2V(f,s){|f|}^2f, \label{Ju1} \\[4mm]  
J(u_2,f,s)&=&\displaystyle 2V(f,s){|f|}^2f-8\frac{V(f,s){(f,f')}^2}{{|f'|}^2}f+4\frac{V(f,s){|f|}^2(f,f')}{{|f'|}^2}f', \>\>\> \label{Ju2} \nonumber \\[4mm]
%
%
J(u_3,f,s)&=&\displaystyle 4\frac{V(f,s){(f,f')}^2}{{|f'|}^2}f-2\frac{V(f,s){|f|}^2(f,f')}{{|f'|}^2}f',  \label{Ju3} \nonumber \\[4mm]
J(u_4,f,s)&=&\displaystyle 4\frac{V(f,s){(f,f')}^2}{{|f'|}^2}f-2\frac{V(f,s){|f|}^2(f,f')}{{|f'|}^2}f'. \label{Ju4} \nonumber
\end{eqnarray}
\end{lemma}
\begin{proof}
Let us prove \eqref{Ju1}. The proof of the other formulae is automatic. Since
\[
\begin{array}{rcl}
u_1(\tilde f,s)&=&\displaystyle 2\Pf\int_{S^1}\frac{\frac{f(t)}{{|f(t)|}^2}-\frac{f(s)}{{|f(s)|}^2}}{\frac{{|f(t)-f(s)|}^4}{{|f(t)|}^4{|f(s)|}^4}} \cdot 
\frac{|f'(t)|}{{|f(t)|}^2}\,dt \\[6mm]
&=&\displaystyle 2\Pf\int_{S^1}\frac{{|f(s)|}^4f(t)-{|f(s)|}^2{|f(t)|}^2f(s)}{{|f(t)-f(s)|}^4}|f'(t)|\,dt \\[4mm]
&=&\displaystyle {|f(s)|}^4\,u_1(f,s)
+2{|f(s)|}^2
\left(\Pf\int_{S^1}\frac{{|f(s)|}^2-{|f(t)|}^2}{{|f(t)-f(s)|}^4}|f'(t)|\,dt\right)f(s),
\end{array}
\]
we have 
\[
\begin{array}{rcl}
J(u_1,f,s)&=&\displaystyle 2{|f(s)|}^2
\left(\Pf\int_{S^1}\frac{{|f(s)|}^2-{|f(t)|}^2+2(f(t)-f(s),f(s))}{{|f(t)-f(s)|}^4}|f'(t)|\,dt\right)f(s) \\[4mm]
&=&\displaystyle 2{|f(s)|}^2
\left(\Pf\int_{S^1}\frac{-|f'(t)|}{{|f(t)-f(s)|}^2}\,dt\right)f(s) \\[4mm]
&=& \displaystyle -2{|f(s)|}^2\,V(f,s)f(s).
\end{array}
\]
\end{proof}
Now \eqref{G_E_u1u4}, Lemma \ref{lemma_cond_MI} and Lemma \ref{lemma_Ju14} imply 
\begin{corollary}
The $V^{-3}$-weighted gradient $\mathcal{G}_E^{V^3}(f)=(V(f))^{-3}G_E(f)$ is \M invariant. 
\end{corollary}
%

\section{Remarks and open problems}
We like to close the article with remarks and open problems. 

\smallskip
(1) The argument for the case when $\Phi=V^3$ can be generalized to the space of embedded odd dimensional closed submanifolds in $\RR^N$ since the potential and the energy have natural generalization. The difficulty in even dimensional case is explained in Subsection 3.3 of \cite{OS2}. Unfortunately, for the moment the author does not know whether an analogue of Proposition \ref{Doyle-Schramm}, i.e. the positivity of the potential for non-spherical submanifolds holds or not. 

\smallskip
(2) If we forget about the positivity condition for $\Phi$, we can use a ``local'' function as $\Phi$ as follows. 
There is a notion of {\em conformal arc-length} $\rho$ given by 
\[
d\rho=\sqrt[4]{{\kappa'}^2+\kappa^2\tau^2}\,ds,
\]
where $s$ is the usual arc-length parameter of the knot $f(S^1)$, $\kappa$ and $\tau$ are the curvature and torsion, and $\kappa'$ is the derivative with respect to the arc-length, $\kappa'=d\kappa/ds$ (\cite{Fi}). 
Since the conformal arc-length $\rho$ is invariant under M\"obius transformations, 
if we denote the arc-length parameter of $T(f(S^1))$ by $\tilde s$, then \eqref{formula_T_tangent_vector} implies 
\[
\frac{d\tilde s}{d\rho}={\left|T'(f(t))\right|}^2\,\frac{ds}{d\rho}.
\]
Therefore, if we put 
\[
\Phi_c(f,t)=\left(\frac{d\rho}{ds}\right)^3
=\left({\Big(\,\frac{d\kappa}{ds}\,\Big)}^2+\kappa^2\tau^2\right)^{3/4} 
\hspace{0.6cm}(s\>\mbox{ is the arc-length}),
\]
then $\Phi_c$ satisfies the condition \eqref{condition_varphi}. 
But $\Phi_c$ is not necessarily positive; it vanishes on a circular or linear arc of a knot. 

\smallskip
(3) Similarly, among the \Mi functionals for knots, the absolute value of writhe (\cite{Ban-Wh}) and the renormalized measure of circles linking a knot (\cite{OS}) do not fit our scheme since the writhe vanishes for planar curves and the latter may be negative.

\end{document}